\documentclass[12pt]{amsart}

\usepackage{amssymb}
\usepackage{amsmath}
\usepackage{enumerate}
\usepackage[latin1]{inputenc}

\usepackage[T1]{fontenc}
\usepackage[sc]{mathpazo}
\usepackage[pdftex]{graphicx, color}
\usepackage{xypic}

\definecolor{alert}{rgb}{0.8,0,0}

\renewcommand{\r}{\mathbb{R}}
\renewcommand{\c}{\mathbb{C}}
\newcommand{\s}{\mathbb{S}}
\newcommand{\h}{\mathbb{H}}
\renewcommand{\r}{\mathbb{R}}

\newcommand{\p}{\mathbb{P}}

\DeclareMathOperator{\Ric}{Ric}

\DeclareMathOperator{\traza}{tr}
\DeclareMathOperator{\arcos}{arcos}

\newtheorem{theorem}{Theorem}

\newtheorem{corollary}{Corollary}
\newtheorem{lemma}{Lemma}

\theoremstyle{definition}

\theoremstyle{remark}
  \newtheorem{remark}{Remark}

\numberwithin{equation}{section}

\begin{document}

\title[On hypersurfaces of $\s^2\times\s^2$]{On hypersurfaces of $\s^2\times\s^2$}

\author{Francisco Urbano}
\address{Departamento de Geometr\'{\i}a  y Topolog\'{\i}a \\
Universidad de Granada \\
18071 Granada, Spain}
\email{furbano@ugr.es}

\thanks{Research partially supported by a MINECO-FEDER grant no. MTM2014-52368-P.}



\date{}
\begin{abstract}
We classify the homogeneous and isoparametric hypersurfaces of $\s^2\times\s^2$. In the classification, besides the hypersurfaces $\s^1(r)\times\s^2,\,r\in (0,1]$, it appears a family of hypersurfaces with three different constant principal curvatures and zero Gauss-Kronecker curvature. Also we classify the hypersurfaces of $\s^2\times\s^2$ with at most two constant principal curvatures and, under certain conditions,  with three constant principal curvatures.
\end{abstract}

\maketitle

\section{Introduction}

Let $(N,g)$ be a compact $4$-dimensional Riemannian manifold and $\Phi:M\rightarrow N$ a two-sided hypersurface. We are interested in the following  properties:

\begin{enumerate}
\item $M$ is (extrinsically) locally homogeneous, i.e., for any points $p,q\in M$ there exist neighbourhoods $V$ and $W$ of $p$ and $q$ and an isometry $F$ of $N$ such that $F(\Phi(V))=\Phi(W)$.

\item $M$ has constant principal curvatures.

\item $M$ is isoparametric, i.e., there exists an isoparametric function $F:N\rightarrow\r$ such that $M=F^{-1}(t)$, for some regular value $t$ of $F$. $F$ is isoparametric if  the gradient and the Laplacian of $F$ satisfy
\[
|\nabla F|^2=f(F),\quad \Delta F=g(F),
\]
where $f,g:\r\rightarrow\r$ are smooth functions.
\end{enumerate}
When $N$ is the $4$-dimensional sphere $\s^4$ or the complex projective plane $\c\p^2$, these properties have been studied and the corresponding classifications have been done (see [C], [K], [M] and [T]). In both cases, the above three properties are equivalent and the number of possible different principal curvatures are $1,2$ or $3$ when $N=\s^4$ and $2$ or $3$ when $N=\c\p^2$.

Besides the above ambient spaces, $\s^2\times\s^2$ is the most interesting compact $4$-manifold to study its hypersurfaces. It is, together with $\c\p^2$, the only compact Hermitian symmetric $4$-manifold. 

In this paper we start the study of the above properties for the  hypersurfaces of $\s^2\times\s^2$. In section 3, we give a complete description of the most important examples, which appear in two families of isoparametric hypersurfaces. The first one, $\{\s^1(r)\times\s^2,\, r\in(0,1]\}$, is a family of homogeneous and isoparametric hypersurfaces with $1$ or $2$ constant principal curvatures.  The second one, $\{M_t,\,t\in (-1,1)\}$, with

$$M_t=\{(p,q)\in\s^2\times\s^2\,|\,<p,q>=t\},$$ is also a family of homogeneous and isoparametric hypersurfaces but with three constant principal curvatures and with Gauss-Kronecker curvature $K=0$. All these examples are tubes over distinguish totally geodesic surfaces of $\s^2\times\s^2$, and, in contrast with the cases of $\s^4$ and $\c\p^2$, the geodesic balls of $\s^2\times\s^2$ do not belong to the above families of examples.

As it is well-known, the Gauss and Codazzi equations  (and hence the curvature of $\s^2\times\s^2$) play an important role in the study of the above properties. In our case, the curvature depends of the product structure of $\s^2\times\s^2$ (see section 2 ) and so the Codazzi equation reflects the behaviour of the hypersurface with respect to the product structure.  This behaviour is described by a function $C$ (see  \eqref{C}) defined on the hypersurface and satisfying $-1\leq C\leq 1$, and so, the properties of this function will be quite important throughout the paper. This function is constant in all the above examples ( $C=1$ for the first family and $C=0$ for the second one).

The first important results in the paper, Theorem~\ref{te:Cconstant} and Corollary~\ref{co:Cconstant}, provide a local characterization of the above examples among the family of hypersurfaces of $\s^2\times\s^2$ where the function $C$ is constant. This characterization will be used  along the paper.

 In Corollary~\ref{co:homogeneous} and Corollary~\ref{co:isoparametric} we prove the following local result, which classifies the homogeneous and isoparametric hypersurfaces of $\s^2\times\s^2$:
\begin{quote}
 
\begin{enumerate}
\item { \it Open subsets of $\{\s^1(r)\times\s^2, \, r\in(0,1]\}$ and $\{M_t,\,t\in(-1,1)\}$ are, up to congruences, the only locally homogeneous orientable hypersurfaces of $\s^2\times\s^2$.
\item $\{\s^1(r)\times\s^2,\, r\in(0,1]\}$ and $\{M_t,\,t\in(-1,1)\}$ are, up to congruences,  the only isoparametric orientable hypersurfaces of $\s^2\times\s^2$}.
\end{enumerate}
\end{quote}
In fact, in Theorem~\ref{te:isoparametric} we prove a stronger result than in (2): we characterize locally the above examples as the only orientable hypersurfaces whose parallel hypersurfaces have constant mean curvature. It is well-known that this property is satisfied by any isoparametric hypersurface.

Finally in  section 6, we study the orientable hypersurfaces of $\s^2\times\s^2$ with constant principal curvatures. In Theorem~\ref{te:2curvaturas} we locally classify  them, when the number of constant principal curvatures is one or two, proving that
\begin{quote}
\begin{enumerate}
\item {\it Up to congruences, open subsets of $\s^1\times\s^2$ are the only orientable hypersurfaces of $\s^2\times\s^2$ with one constant principal curvature.
\item Up to congruences, open subsets of $\{\s^1(r)\times\s^2,\, r\in(0,1)\}$, are the only orientable hypersurfaces of $\s^2\times\s^2$ with two different constant principal curvatures}.
\end{enumerate}
\end{quote}

When the number of different principal curvatures is three, the classification problem is harder, and we have only got partial results. Using Theorem~\ref{te:3}, where we study the critical points of the function $C$ in such hypersurfaces, we prove in Corollary~\ref{co:compact} the following result:
\begin{quote}
{\it $\{M_t,\,t\in(-1,1)\}$ are, up to congruences, the only orientable compact hypersurfaces with three different constant principal curvatures,  with scalar curvature $\rho\not =1/2$ and Gauss-Kronecker curvature $K=0$}.
\end{quote}

\section{Preliminaries}
Let $\s^2$ be the $2$-dimensional unit sphere, $\langle,\rangle$ its standar metric and $J$ its complex structure defined by
\[
J_pv=p\wedge v,\quad p\in\s^2,\,v\in T_p\s^2.
\]
 We endow $\s^2\times\s^2$ with the product metric (also denoted by $\langle,\rangle$) and the complex structures
\[
J_1=(J,J),\quad J_2=(J,-J)
\]
which define two structures of K\"ahler surface on $\s^2\times\s^2$. It is clear that, if $Id:\s^2\rightarrow\s^2$ is the identity map and $F:\s^2\rightarrow\s^2$ is any anti-holomorphic isometry of $\s^2$, then  $Id\times F:(\s^2\times\s^2, J_1)\rightarrow (\s^2\times\s^2, J_2)$ is a holomorphic isometry.

The product structure $P$ on $\s^2\times\s^2$,  defined by 
\[
P(v_1,v_2)=(v_1,-v_2),\quad v_1,v_2\in T\s^2,
\]
satisfies $P=-J_1J_2=-J_2J_1$ and $\bar{\nabla}P=0$, where $\bar{\nabla}$ is the Levi-Civita connection on $\s^2\times\s^2$.

On the other hand, using that $\s^2\times\s^2$ is a product manifold, its curvature tensor $\bar{R}$ is given by
\begin{gather*}
\bar{R}(v,w,x,y)=\frac{1}{2}\{\langle v,y\rangle\langle w,x\rangle-\langle v,x\rangle\langle w,y\rangle\\
+\langle Pv,y\rangle\langle Pw,x\rangle-\langle Pv,x\rangle\langle Pw,y\rangle\},
\end{gather*}
where $v,w,x,y\in T(\s^2\times\s^2)$,
and hence $\s^2\times\s^2$ is an Einstein manifold with scalar curvature $4$ and non-negative sectional curvature.

Finally, the group of isometries of $\s^2\times\s^2$ is the $6$-dimensional subgroup of the orthogonal group $O(6)$ given by
\begin{equation}\label{isom}
\left\{ \left( \begin{array}{cc} A & 0 \\ 0 & B \end{array} \right) \,,\,\left( \begin{array}{cc} 0 & A \\ B & 0 \end{array} \right)\,/\, A,B \in \hbox{O}(3) \right\}.
\end{equation}

Let $\Phi:M\rightarrow\s^2\times\s^2$ be an orientable hypersurface of $\s^2\times\s^2$ and $N$ a unit normal vector field to $\Phi$. The behaviour of $M$ with respect  to the product structure is given by the smooth function $C$ and the vector field $X$ tangent to $M$ defined by
\begin{equation}\label{C}
\begin{split}
C:M\rightarrow  \r,\quad\quad
 &C=\langle PN,N\rangle=\langle J_1N,J_2N\rangle,\\
X=PN-CN.
\end{split}
\end{equation}

It is clear that $-1\leq C\leq 1$, that $X$ is the tangential component of $PN$ and that $|X|^2=1-C^2$.

From the Gauss equation it follows that the scalar curvature $\rho$ of $M$ is given by
\[
\rho=2+9H^2-|\sigma|^2,
\]
where $H$ is the mean curvature vector field and $\sigma$ the second fundamental form of $\Phi$.
The Codazzi equation is given by
\[
(\nabla\sigma)(v,w,x)-(\nabla\sigma)(w,v,x)=\frac{1}{2}\big(\langle X,v\rangle\langle Pw,x\rangle-\langle X,w\rangle\langle Pv,x\rangle\big),
\]
where $\nabla\sigma$ is the covariant derivative of the second fundamental form.

In the following result we describe some properties of $C$ and $X$ which will be used along the paper.

\begin{lemma}\label{le:C,X}
Let $\Phi: M\rightarrow \s^2\times\s^2$ be an orientable hypersurface and $A$ the shape operator associated to the unit normal field $N$. Then
\begin{enumerate}
\item The gradient of $C$ and the covariant derivative of $X$ are given by
\[
\nabla C=-2AX,\quad  \nabla_VX=C AV- P^{T}AV,
\]
\item The Hessian of $C$ is given by
\begin{eqnarray*}
(\nabla^2C)(V,W)=
-2(\nabla\sigma)(V,X,W)-2C\langle AV,AW\rangle+2\langle PAV,AW\rangle.
\end{eqnarray*}
\item The Laplacian of $C$ and the divergence of $X$ are given by
\[
\Delta C=-6\langle X,\nabla H\rangle-2C|\sigma|^2+2\traza (P^TA^2),\quad \hbox{div}\,X=3CH-\traza (P^TA),
\]
\end{enumerate}
where $P^T:TM\rightarrow TM$ is the tangential component of the restriction of $P$ to $M$, $\traza$ stands for the trace and $V,W$ are vector fields on $M$.
\end{lemma}
\begin{proof}
Derivating the second equation of \eqref{C} and taking into account that $P$ is parallel,  we get easily (1). Now, (2) and (3) follow easily from (1) using the Codazzi's equation.
\end{proof}
Finally, if $\Phi:M\rightarrow\s^2\times\s^2$ is an orientable hypersurface and $\{e_1,e_2,e_3\}$ is an orthonormal reference of $M$ such that $\{e_1,e_2,e_3,N\}$ is positively oriented and  $P_{ij}=\langle Pe_i,e_j\rangle$, $b_i=\langle Pe_i,N\rangle=\langle X,e_i\rangle$, then the product structure $P$, in the above reference, is written as follows
\[
P=\left( \begin{array}{cccc} P_{11} & P_{12} & P_{13} & b_1 \\ P_{21} & P_{22} & P_{23} & b_2 \\ P_{31} & P_{32} & P_{33} & b_3 \\ b_1 & b_2 & b_3 & C \end{array} \right) 
\]
As $P\in SO(4)$, $P=P^{t}$ and $\traza\,P=0$, it follows that for $i\not=j\not=k$
\begin{equation}\label{P}
\begin{split}
CP_{ii}-b_i^2=P_{jj}P_{kk}-P_{jk}^2,\quad CP_{ij}-b_ib_j=P_{ik}P_{jk}-P_{ij}P_{kk},\\
b_iP_{ij}-b_jP_{ii}=-b_kP_{kj}+b_jP_{kk}.
\end{split}
\end{equation}

\section{Examples}
In this section we are going to give the most regular examples of hypersurfaces of $\s^2\times\s^2$, some of them will be characterized in the paper.

\subsection{Hypersurfaces with function $C$ satisfying $C^2=1$.}
Given  $a\in\s^2$, let $G:\s^2\times\s^2\rightarrow \r$ be the function defined by
\[
G(p,q)=\langle p,a\rangle.
\]
Then it is easy to check that the gradient and the Laplacian of G satisfy
\[
|\bar{\nabla} G|^2=1-G^2,\quad \bar{\Delta} G=-2G.
\]
This means that $G$ is an isoparametric function on $\s^2\times\s^2$ and hence the level hypersurfaces of $G$ define a one-parameter family of hypersurfaces of $\s^2\times\s^2$ with constant mean curvature.  

In this particular case,  $G^{-1}(t)=\emptyset$  if $|t|>1$,  $G^{-1}(1)=\{a\}\times\s^2$ and $G^{-1}(-1)=\{-a\}\times\s^2$ are the focal sets, which are totally geodesic surfaces of $\s^2\times\s^2$. Finally, for $t\in(-1,1)$ we have that
\[
G^{-1}(t)=\{(p,q)\in\s^2\times\s^2\,|\,\langle p,a\rangle=t\} 
\]
is a hypersurface of $\s^2\times\s^2$ with constant mean curvature. The isometry  of $\s^2\times\s^2$ given by $-Id\times Id$ transforms $G^{-1}(-t)$ onto $G^{-1}(t)$. 
Also, it is clear that, up to congruences, we can take $a=(0,0,1)$.
So we have a family of hypersurfaces
\[
G^{-1}(t)= \s^1(r)\times\s^2,\quad r^2=1-t^2,\, t\in[0,1),\,r\in(0,1],
\]
where $\s^1(r)=\{(x,y,\sqrt{1-r^2})\in\s^2\}$.
It is trivial to check that $G^{-1}(0)=\s^1\times\s^2$ is totally geodesic and that $\s^1(r)\times\s^2,\, r\in (0,1)$, has two constant principal curvatures: $0$ with multiplicity two and $\frac{\sqrt{1-r^2}}{r}$ with multiplicity one.
 
Also, $\{\s^1(r)\times\s^2,\,r\in(0,1]\}$ are tubes of radius $\arcos\sqrt{1-r^2}$ over the focal surface $\{a\}\times\s^2$, with $a=(0,0,1)$.

Finally, the group of isometries of $\s^2\times\s^2$ given by
\[
\left\{ \left( \begin{array}{cc} A & 0 \\ 0 & B \end{array} \right) \,/\, A=\left( \begin{array}{cc} \hat{A} & 0 \\ 0 & 1 \end{array} \right), \hat{A}\in \hbox{SO}(2)\,,\,B\in\hbox{SO(3)}\right\}
\]
acts transitively on $\s^1(r)\times\s^2$ and hence these hypersurfaces are homogeneous. 
Sumarizing, we have that
\begin{quote}
{\it $\{\s^1(r)\times\s^2,\, r\in(0,1]\}$ is a family of homogeneous isoparametric hypersurfaces of $\s^2\times\s^2$ with two constant principal curvatures when $r\in (0,1)$ and totally geodesic when $r=1$.} 
\end{quote}
These hypersurfaces satisfy that $C= 1$, because the unit normal field has no component in the second factor. We remark that the isometry of $\s^2\times\s^2$ given by $(p,q)\mapsto (q,p)$ transforms $\s^1(r)\times\s^2$ onto $\s^2\times\s^1(r)$ whose function $C= -1$.

Now we are going to characterize locally the hypersurfaces satisfying $C^2= 1$. Without loss of generality we can assume that $C=1$.

If $\Phi=(\phi,\psi):M\rightarrow \s^2\times\s^2$ is an orientable hypersurface with $C=1$, then $X=0$, $PN=N$ and $J_1N=J_2N$. Hence, from Lemma~\ref{le:C,X}  we get that $A=PA$.
Moreover the tangent bundle decomposes as $TM=<J_1N>\oplus D$, where $D$ is the two-dimensional distribution orthogonal to $J_1N$. As $P(J_1N)=J_1N$, it is clear that $P_{|D}=-Id$, and so $A_{|D}=0$. Also, if $V,W$ are vector fields on $D$, we have that
\[
\langle \nabla_VW,J_1N\rangle=-\langle W,\bar{\nabla}_VJ_1N\rangle=\langle W,J_1(AV)\rangle=0,
\]
which means that $D$ is a totally geodesic foliation on $M$.  If $\Sigma$ is a leaf of $D$, it follows that $\psi:\Sigma\rightarrow \s^2$ is a local isometry, and hence 
\begin{quote}
{\it Any hypersurface of $\s^2\times\s^2$ with $C\equiv 1$ is locally the product of a integral curve of $J_1N$ in $\s^2$ and an open subset of $\s^2$}.
\end{quote}

\subsection{Hypersurfaces with three constant principal curvatures}
Let $F:\s^2\times\s^2\rightarrow\r$ be the function defined by
\[
F(p,q)=\langle p,q\rangle.
\]
Then it is not difficult to check that
\[
|\bar{\nabla} F|^2=2(1-F^2),\quad \bar{\Delta} F=-4F,
\]
and so $F$ is an isoparametric function on $\s^2\times\s^2$. Hence the level hypersurfaces of $F$ have constant mean curvature. In this case, $F^{-1}(t)$ is empty if $|t|>1$, and  the diagonal surface $F^{-1}(1)=\{(p,p)\in\s^2\times\s^2\}$ and the anti-diagonal surface $F^{-1}(-1)=\{(p,-p)\in\s^2\times\s^2\}$ of $\s^2\times\s^2$ are the focal sets of $F$.

For $t\in(-1,1)$ we have that 
\[
M_t= \{(p,q)\in\s^2\times\s^2\,|\,\langle p,q\rangle=t\}
\]
is an hypersurface of $\s^2\times\s^2$ with constant mean curvature.

The hypersurfaces $M_t$ and $M_{-t}$ are congruents because the isometry $I$ of $\s^2\times\s^2$ given by $I(p,q)=(p,-q)$ transforms $M_t$ onto $M_{-t}$.

Moreover,  the tube of radius $\arccos(t/\sqrt 2)$ over the diagonal surface $F^{-1}(1)=\{(p,p)\in\s^2\times\s^2\}$ is given by the sets of points $\{(x,y)\in\s^2\times\s^2\}$ such that

\begin{gather*}
(x,y)=\big(\cos (\frac{t}{\sqrt2})p+\sqrt2\sin(\frac{t}{\sqrt2})v,\,\cos (\frac{t}{\sqrt2})p-\sqrt2\sin(\frac{t}{\sqrt2})v\big),\\
p\in\s^2,\,v\in T_p\s^2,\,|v|=1/\sqrt2.
\end{gather*}
As $\langle x,y\rangle=\cos^2\frac{t}{\sqrt2}-\sin^2\frac{t}{\sqrt2}=\cos(\sqrt2 t)$, we obtain that the hypersurface $M_t$ is a tube of radius 
$\arccos(t/\sqrt 2)$ over the diagonal surface.

On the other hand, it is clear that $SO(3)$ acts transitively by isometries on $M_t$ by
\[
A (p,q)=(Ap,Aq), \quad A  \in SO(3),
\]
and hence $\{ M_t,\,t\in(-1,1)\}$ is a family of homogeneous hypersurfaces.

Also, the isotropy subgroup of the above action at any point of $M_t$ is the identity. So $M_t$  is diffeomorphic to $SO(3)\equiv \r\p^3$. Hence $M_t$ is a  homogeneous Riemannian manifold  and  $SO(3)$
is the group of isometries of $M_t$ when $t\not=0$ and that $SO(3)$ joint with the one-parameter group of isometries
$\{h_t:M_0\rightarrow M_0,\,t\in\mathbb{R}\}$ defined by
\[
h_t(p,q)=(tp+\sqrt{1-t^2}q, \sqrt{1-t^2}p-tq),
\]
is the group of isometries of $M_0$. We remark that $\{h_t\}$ is only well-defined on $M_0$ and that they are the restriction to $M_0$ of isometries of $O(6)$, which no define isometries of $\s^2\times\s^2$.

In [MP], the simply connected homogeneous Riemannian three-manifolds are described in detail. Following its notation, $M_0$ is the Berger projective space with $\kappa =1$ and $\tau^2 =1/2$. Also, $M_t,\,t\not=0$ is the projective space with the metric given by the parameters $c_1=2=c_2+c_3$ with $c_2=1+t$ and $c_3=1-t$.

Now, we are going to study more properties of the hypersurfaces $\{M_t\}$. It is easy to check that 
$$N_{(p,q)}=\frac{1}{\sqrt{2(1-t^2)}}(q-tp,p-tq)$$
 is a unit normal vector field to $M_t$ in $\s^2\times\s^2$ and so we have that these hypersurfaces have the function $C$ constantly zero.  Hence, if $J_i,\,i=1,2$, and $P$ are the complex structures and the product structure on $\s^2\times\s^2$, then
\begin{gather*}
J_1N_{(p,q)}=\frac{(p\wedge q,-p\wedge q)}{\sqrt{2(1-t^2)}},\quad J_2N_{(p,q)}=\frac{(p\wedge q,p\wedge q)}{\sqrt{2(1-t^2)}},\\
X=  PN_{(p,q)} =\frac{1}{\sqrt{2(1-t^2)}}(q-tp,-p+tq),
\end{gather*}
is a trivialization of $M_t$ by orthonormal vectors fields, where $\wedge$ stands for the vectorial product in $\mathbb{R}^3$. If $A$ denotes the shape operator associated to $N$, then for any  vector $(v_1,v_2)$ tangent to $M_t$, we have that
\[
A(v_1,v_2)=\frac{1}{\sqrt{2(1-t^2)}}\left\{t(v_1,v_2)-(v_2,v_1)+\langle p,v_2\rangle(p,-q)\right\}.
\]
 From here we obtain that
\[
A(J_1N)=\frac{1}{\sqrt2}\sqrt{\frac{1+t}{1-t}}\,J_1N,\quad A(J_2N)=-\frac{1}{\sqrt2}\sqrt{\frac{1-t}{1+t}}\,J_2N,\quad A\,X=0.
\]
So, $\{M_t,\,t\in(-1,1)\}$ are hypersurfaces of $\s^2\times\s^2$ with three constant principal curvatures, $\lambda_1=0,\, \lambda_2\lambda_3=-1/2$. So the Gauss-Krocneker curvature of $M_t$ is  zero. The lengths of the mean curvature vector field and the second fundamental form are given by
\[
H=\frac{\sqrt{2}t}{3\sqrt{1-t^2}},\quad |\sigma|=\sqrt{\frac{1+t^2}{1-t^2}}.
\]
Among all the $M_t$, only $M_0$ is minimal. From the Gauss equation we obtain that the sectional curvature, the Ricci tensor and the scalar curvature  of $M_t$ satisfy the following properties:
\begin{gather*}
K(J_1N\wedge J_2N)=-\frac{1}{2},\quad K(J_1N\wedge X)=K(J_2N\wedge X)=\frac{1}{2},\\
Ric(v)=\langle v,X\rangle^2\geq 0,\quad \rho=1.
\end{gather*}
We remark that the curvatures of $M_t$ do not depend of $t$. 
Sumarizing, we have that
\begin{quote}
{\it $$\{M_t=\{(x,y)\in\s^2\times\s^2\,|\,\langle x,y\rangle=t\},\, t\in (-1,1)\}$$ is a family of homogeneous isoparametric hypersurfaces of $\s^2\times\s^2$ with three  constant principal curvatures. Their Gauss-Krocneker curvatures vanish and only $M_0$ is a minimal hypersurface. Moreover, all these hypersurfaces have  $C=0$. } 
\end{quote}
\begin{remark}
The above examples can be defined in higher dimension. In fact,  if $\s^n$ is the $n$-dimensional unit sphere with its canonical metric and in $\s^n\times\s^n$ we consider the product metric, then 
\[
M_t=\{(p,q)\in\s^n\times\s^n\,|\, \langle p,q\rangle=t\},\quad t\in(-1,1)
\]
define a one-parameter family of homogeneous isoparametric hypersurfaces of $\s^n\times\s^n$ with three constant principal curvatures: $0$ with multiplicity one and $\frac{\sqrt{1+t}}{\sqrt{2(1-t)}}, -\frac{\sqrt{1-t}}{\sqrt{2(1+t)}}$ with multiplicities $n-1$.  
 \end{remark}

\subsection{Other interesting examples}

1) Given $a,b\in \s^2$ we define
\[
M_{a,b}=\{(p,q)\in\s^2\times\s^2\,|\,\langle p,a\rangle+\langle q,b\rangle=0\}.
\]
Then it is easy to check that $M_{a,b}$ is a compact hypersurface of $\s^2\times\s^2$ with two isolated singularities: $(a,-b),(-a,b)$. Outside of these singularities, a unit normal vector field to $M_{a,b}$ is defined by
\[
N(p,q)=\frac{1}{\sqrt{2(1-\langle p,a\rangle^2)}}(a-\langle p,a\rangle p,b-\langle q,b\rangle q),
\]
and hence the function $C=\langle PN,N\rangle=0$.

Also it is straighforward to see that the orthonormal reference $\{X, E_2=(J_1N+J_2N)/\sqrt{2}, E_3=(J_1N-J_2N)/\sqrt{2}\}$ diagonalizes the second fundamental form with
\[
AX=0,\quad AE_2=\frac{\langle p,a\rangle}{\sqrt{2(1-\langle p,a\rangle^2)}}E_2,\quad AE_3=\frac{-\langle p,a\rangle}{\sqrt{2(1-\langle p,a\rangle^2)}}E_3.
\]
Hence $M_{a,b}$ is a minimal hypersurface with Gauss-Kronecker curvature $K=0$ and  with scalar curvature $$\rho(p,q)=\frac{2-3\langle p,a\rangle^2}{1-\langle p,a\rangle^2},\quad -\infty<\rho\leq 2.$$
A parametrization of $M_{a,b}$ when $a=(0,0,1),\, b=(0,0,-1)$,  is given by the triply periodic ramified immersion
\begin{gather*} 
\Phi:\r^3\rightarrow \s^2\times\s^2\\
\Phi(t,r,s)=\cos(\frac{t}{\sqrt 2})\big((\cos r,\sin r,0),(\cos s,\sin s,0)\big)+\sin (\frac{t}{\sqrt 2})(a,-b).
\end{gather*}
2) Given $a,b\in \s^2$ we define
\[
\hat{M}_{a,b}=\{(p,q)\in\s^2\times\s^2\,|\,\langle p,a\rangle^2+\langle q,b\rangle^2=1\}.
\]
Then it is easy to check that $\hat{M}_{a,b}$ is a compact hypersurface of $\s^2\times\s^2$ with four curves of singularities
\[
\{(p,\pm b)\,|\,\langle p,a\rangle=0\},\quad \{(\pm a,q)\,|\,\langle q,b\rangle=0\}.
\]
Outside of these singularities, a unit normal vector field to $\hat{M}_{a,b}$ is defined by
\[
N(p,q)=\frac{1}{\sqrt{2}\,|\langle p,a\rangle\langle q,b\rangle|}\big(\langle p,a\rangle(a-\langle p,a\rangle p),\langle q,b\rangle(b-\langle q,b\rangle q)\big),
\]
and hence the function $C=\langle PN,N\rangle=0$.

Also it is straighforward to see that the orthonormal reference $\{X, E_2=(J_1N+J_2N)/\sqrt{2}, E_3=(J_1N-J_2N)/\sqrt{2}\}$ diagonalizes the second fundamental form with
\[
AX=0,\quad AE_2=\frac{\langle p,a\rangle^2}{\sqrt{2}\,|\langle p,a\rangle\langle q,b\rangle|}E_2,\quad AE_3=\frac{\langle q,b\rangle^2}{\sqrt{2}\,|\langle p,a\rangle\langle q,b\rangle|}E_3.
\]
Hence $\hat{M}_{a,b}$ is a hypersurface with Gauss-Kronecker curvature $K=0$, with constant curvature $1/2$ and the length of the mean curvature vector field is given by
\[
H(p,q)=\frac{1}{3\sqrt{2}\,|\langle p,a\rangle\langle q,b\rangle|}.
\]
A parametrization of $\hat{M}_{a,b}$, when $a=b=(0,0,1)$, is given by the triply periodic ramified immersion
\begin{gather*}
\Phi=(\phi,\psi): \r^3\rightarrow\s^2\times\s^2\\
\phi(t,r,s)=\frac{\cos(\frac{t}{\sqrt{2}})-\sin(\frac{t}{\sqrt{2}})}{\sqrt{2}}(\cos r,\sin r,0)+(0,0,\frac{\cos(\frac{t}{\sqrt{2}})+\sin(\frac{t}{\sqrt{2}})}{\sqrt{2}}),\\
\psi(t,r,s)=\frac{\cos(\frac{t}{\sqrt{2}})+\sin(\frac{t}{\sqrt{2}})}{\sqrt{2}}(\cos s,\sin s,0)+(0,0,\frac{\cos(\frac{t}{\sqrt{2}})-\sin(\frac{t}{\sqrt{2}})}{\sqrt{2}}).
\end{gather*}
\section{Characterizations of the examples. Homogeneous hypersurfaces}

In the next result we give  two local characterizations of the examples defined in section 3.

\begin{theorem}\label{te:Cconstant}
Let $\Phi:M\rightarrow\s^2\times\s^2$ be an orientable hypersurface with $C=\langle PN,N\rangle$ constant, where $N$ is a unit normal field to $\Phi$.  Then
\begin{enumerate}
\item $M$ has constant mean curvature if and only if either
\begin{enumerate}
\item $C^2=1$ and $\Phi(M)$ is congruent to an open set of $\s^1(r)\times\s^2$ for some $r\in(0,1]$,
\item or $C=0$ and $\Phi(M)$ is congruent to an open set of $M_t$ for some $t\in(-1,1)$,
\item or $C=0$ and $M$ is a non-compact minimal hypersurface with non-constant scalar curvature.
\end{enumerate}

\item $M$ has constant scalar curvature if and only if either
\begin{enumerate}
\item $C^2=1$ and $\Phi(M)$ is congruent to an open set of $\s^1(r)\times\s^2$ for some $r\in(0,1]$,
\item or $C=0$ and $\Phi(M)$ is congruent to an open set of $M_t$ for some $t\in(-1,1)$,
\item or $C=0$ and $M$ is a non-complete hypersurface with constant curvature $1/2$ and non-constant mean curvature.
\end{enumerate}
\end{enumerate}
\end{theorem}

\begin{remark} 
\begin{enumerate}
\item The family of minimal hypersurfaces in item (1c)  is not empty, because the hypersurface $M_{a,b}$ given in section 3.3 is a non-complete minimal hypersurface with non-constant scalar curvature and with $C=0$.

\item The family of hypersurfaces in item (2c) is also not empty, because  the hypersurface $\hat{M}_{a,b}$ given in section 3.3 is a non-complete hypersurface with constant curvature $1/2$ , non-constant mean curvature and with $C=0$.

\end{enumerate}
\end{remark}
\begin{corollary}\label{co:Cconstant}

\begin{enumerate}

\item $\{\s^1(r)\times\s^2,\, r\in(0,1]\}$,  $\{M_t,\,t\in(-1,1)\}$ and their compact coverings are the only compact orientable constant mean curvature hypersurfaces of $\s^2\times\s^2$ with $C$ constant. 
\item $\{\s^1(r)\times\s^2,\,r\in(0,1]\}$, $\{M_t,\, t\in(-1,1)\}$ and their compact coverings are the only complete orientable hypersurfaces with constant scalar curvature of  $\s^2\times\s^2$ with $C$ constant.
\item Open subsets of $\{\s^1(r)\times\s^2,\; r\in(0,1]\}$ and $\{M_t,\; t\in(-1,1)\},$ are the only orientable hypersurfaces of $\s^2\times\s^2$ which have the mean curvature, the scalar curvature and the function $C$ constants. 
\end{enumerate}
\end{corollary}
\begin{proof}
Taking into account  section 3, the sufficient conditions in (1) and (2)  are clear.

In order to prove the neccesary conditions, first we suppose that $C^2=1$, and without loss of generality we consider $C=1$. In Section 3.1, we prove that  $M$ is locally congruent to the product of a curve in $\s^2$ and an open subset of $\s^2$. If the mean curvature or the scalar curvature of the hypersurface is constant, then    the curvature of the curve of $\s^2$ is also constant and so 
we obtain the case (1a) or (2a).

Now we suppose that $C=c_0\in (-1,1)$.  Then from Lemma~\ref{le:C,X} (1), it follows that $AX=0$ with $|X|^2=1-c_0^2>0$. So, at any point of $M$,  $0$ is a principal curvature of the hypersurface with corresponding eigenvector $X$. Hence, on $M$ we can consider the orthonormal reference $\{E_i,\,i=1,2,3\}$ where
\[
E_1=\frac{X}{\sqrt{1-c_0^2}},\quad E_2=\frac{J_1N+J_2N}{\sqrt{2(1+c_0)}},\quad E_3=\frac{J_1N-J_2N}{\sqrt{2(1-c_0)}}.
\]  
Using \eqref{C},  the shape operator $A$ and the tangential component of the product structure $P^{T}$ are given, with respect to this reference, by
\[
A=\left( \begin{array}{ccc} 0 &0 & 0 \\ 0 & \sigma_{22} & \sigma_{23}  \\ 0 & \sigma_{23} & \sigma_{33} \end{array} \right),\quad P^{T}=\left( \begin{array}{ccc} -c_0 &0 & 0 \\ 0 & 1 & 0  \\ 0 & 0 & -1
 \end{array} \right).
\]

Using Lemma~\ref{le:C,X} and that $J_i,\,i=1,2,$ are Kähler structures on $\s^2\times\s^2$, i.e. $\bar{\nabla}J_i=0$, it is not difficult to check that the Levi-Civita connection $\nabla$ of the induced metric  on $M$ is given by  
\begin{gather*}
\nabla_{E_1}E_i=0, \quad \nabla_{E_2}E_3=-\sqrt{1-c_0^2}\sigma_{23}E_1,\quad \nabla_{E_3}E_2=\sqrt{1-c_0^2}\sigma_{23}E_1,\\
\nabla_{E_2}E_1=-\sqrt{\frac{1-c_0}{1+c_0}}\sigma_{22}E_2+\sqrt{\frac{1+c_0}{1-c_0}}\sigma_{23}E_3,\\
\nabla_{E_3}E_1=-\sqrt{\frac{1-c_0}{1+c_0}}\sigma_{23}E_2+\sqrt{\frac{1+c_0}{1-c_0}}\sigma_{33}E_3,\\
\nabla_{E_2}E_2=\sqrt{\frac{1-c_0}{1+c_0}}\sigma_{22}E_1,\quad \nabla_{E_3}E_3=-\sqrt{\frac{1+c_0}{1-c_0}}\sigma_{33}E_1.
\end{gather*}

The knowledge of the Levi-Civita connection and the Codazzi equation, joint with  Lemma~\ref{le:C,X}, allow us to get the derivatives of the second fundamental form, obtaining 

\begin{equation}\label{derivatives}
\begin{split}
X(\sigma_{22})=\frac{1-c_0^2}{2}+(1-c_0)\sigma_{22}^2-(1+c_0)\sigma_{23}^2,\\
X(\sigma_{33})=\frac{c_0^2-1}{2}+(1-c_0)\sigma_{23}^2-(1+c_0)\sigma_{33}^2,\\
X(\sigma_{23})=(1-c_0)\sigma_{22}\sigma_{23}-(1+c_0)\sigma_{33}\sigma_{23},\\
E_2(\sigma_{33})=E_3(\sigma_{23}),\quad E_3(\sigma_{22})=E_2(\sigma_{23}).
\end{split}
\end{equation}

{\it Case (1):  the mean curvature $H$ is constant}.

In this case,  from Lemma~\ref{le:C,X},(3) it follows that $c_0|\sigma|^2=\hbox{trace}\,P^{T}A^2$. Using the above reference, this equation becomes in 
\begin{equation}\label{H-constant}
c_0|\sigma|^2=3H(\sigma_{22}-\sigma_{33}).
\end{equation}

First we are going to prove that $c_0=0$. In fact, if $H=0$, as the hypersurface cannot be totally geodesic because in such case $C^2=1$,  the equation \eqref{H-constant} says that $c_0=0$. If $H\not=0$, derivating \eqref{H-constant} with respect to $X$ and using \eqref{derivatives} and \eqref{H-constant} it is straightforward to get that 

\begin{equation}\label{sigma}
|\sigma|^2=\frac{9H^2(1+9H^2)}{c_0^2+9H^2}.
\end{equation}
In particular $|\sigma|^2$ is constant and from \eqref{H-constant} , the function $\sigma_{22}-\sigma_{33}$ is also constant. This implies, taking into account that $H$ and $|\sigma|^2$ are constant functions,  that all the functions  $\sigma_{ij}$ are constants. Using in \eqref{derivatives} that  $X(\sigma_{23})=0$ and \eqref{H-constant} we get 
\[
c_0\sigma_{23}(9H^2-|\sigma|^2)=0.
\]
If  $9H^2-|\sigma|^2=0$, equation \eqref{sigma} says that $c_0^2=1$, which is imposible. If $\sigma_{23}=0$, using that $X(\sigma_{22})=0$ in \eqref{derivatives}, it follows that $c_0=1$, which is imposible. Hence last equation says that $c_0=0$ again.

Hence we have proved that $c_0=0$ and so \eqref{H-constant} says that either $M$ is minimal, i.e. $H=0$, or $\sigma_{22}=\sigma_{33}$. We are going to study these cases separately.

{\it First case: $\sigma_{22}=\sigma_{33}$}. In this case we have that $3H=2\sigma_{22}$ and hence $\sigma_{22}$ and $\sigma_{33}$ are constant functions. Using this in \eqref{derivatives} we obtain that $\sigma_{23}^2=1/2+\sigma_{22}^2$ and so $\sigma_{23}$ is also constant and the hypersurface has constant scalar curvature $\rho=1$.
 
Now, taking into account that $c_0=0$, the second fundamental form, with respect to the orthonormal reference on $M$ given by $\{X,J_1N,J_2N\}$, is given by
\[
AX=0,\quad A(J_1N)=(\sigma_{22}+\sigma_{23})J_1N,\quad A(J_2N)=(\sigma_{22}-\sigma_{23})J_2N,
\]
with $(\sigma_{22}+\sigma_{23})(\sigma_{22}-\sigma_{23})=-1/2$.  As these principal curvatures are constant and their product is $-1/2$, these numbers can be written, without loss of generality, as 
\[
\sigma_{22}+\sigma_{23}=\frac{1}{\sqrt2}\sqrt{\frac{1+t}{1-t}},\quad \sigma_{22}-\sigma_{23}=-\frac{1}{\sqrt2}\sqrt{\frac{1-t}{1+t}}
\]
for certain $0\leq t<1$.

Now to find the focal set of $M$, we consider the parallel hypersurfaces to $M$.  As $C=0$, we have that $N=(N_1,N_2)$ with $|N_1|^2=|N_2|^2=1/2$. Hence, the parallel hypersurfaces to $M$ are given by
$\Phi_s:M\rightarrow\s^2\times\s^2$, $s\geq 0,$  where
\[
\Phi_s(p,q)=(\exp_p(sN_1),\exp_q(sN_2))=\cos(\frac{s}{\sqrt2})(p,q)+\sqrt2\sin(\frac{s}{\sqrt2})N.
\]
Then
\begin{gather*}
(\Phi_s)_*(X)=\cos(\frac{s}{\sqrt2})X-\frac{1}{\sqrt2}\sin(\frac{s}{\sqrt2})\hat{\Phi}\\
(\Phi_s)_*(J_1N)=\big(\cos(\frac{s}{\sqrt2})-\sqrt{\frac{1+t}{1-t}}\sin(\frac{s}{\sqrt2})\big)J_1N,\\
(\Phi_s)_*(J_2N)=\big(\cos(\frac{s}{\sqrt2})+\sqrt{\frac{1-t}{1+t}}\sin(\frac{s}{\sqrt2})\big)J_2N,
\end{gather*}
where $\hat{\Phi}=S\circ\Phi$, $S$ being the isometry of $\s^2\times\s^2$ given by $S(p,q)=(p,-q)$.

Hence, the focal surface of $M$ happens when $\cot(\frac{s}{\sqrt2})= \sqrt{\frac{1+t}{1-t}}$, i.e., when $\cos(\frac{s}{\sqrt2})=\sqrt{(1+t)/2},\,\sin(\frac{s}{\sqrt2})=\sqrt{(1-t)/2}$. But this means that $\cos(\sqrt{2}s)=t$, and so $s=\frac{1}{\sqrt2}\arcos t$.

It $\Sigma$ is the focal surface of $M$  and we denote by $\Psi$ the restriction of $\Phi_{s}$ (with $s=\frac{1}{\sqrt2}\arcos t$) to $\Sigma$, then the immersion $\Psi:\Sigma\rightarrow\s^2\times\s^2$ is given by
\[
\Psi=\frac{\sqrt{1+t}}{\sqrt2}\,\Phi+\sqrt{1-t}\,N.
\]
As $(\Phi_s)_*(J_1N)=0$, for $s=\frac{1}{\sqrt2}\arcos t$, it is clear that $\{X, \,J_2N\}$ is an orthonormal reference of the tangent bundle to $\Sigma$ and that

\begin{gather*}
(\Psi)_*(X)=\frac{\sqrt{1+t}}{\sqrt2}\,X-\frac{\sqrt{1-t}}{2}\,\hat{\Psi},\\
(\Psi)_*(J_2N)=\frac{\sqrt2}{\sqrt{1+t}}\,J_2N.
\end{gather*}
Hence,  $\{J_1N,\, \frac{\sqrt{1-t}}{2}\,\Phi-\frac{\sqrt{1+t}}{\sqrt2}\,N\}$
is an orthonormal reference on the normal bundle of $\Psi$. Now, it is easy to check that the corresponding Weingarten endomorphisms associated to these two unit normal vector fields vanish, and so $\Psi$ is a totally geodesic immersion. Moreover, as $J_1X=J_2N$, the immersion $\Psi$ is a complex surface with respect to the complex structure $J_1$ and a Lagrangian surface with respect to the other complex structure $J_2$. From [CU], we have that $\Psi$ is congruent to an open subset of the diagonal surface $\{(p,p)\in\s^2\times\s^2\,|\, p\in\s^2\}$ and our hypersurface $M$ is an open subset of the tube of radius $s=\frac{1}{\sqrt2}\arcos t$ over the diagonal surface. Taking into account section 3.2,  we get that $M$ is locally congruent to some $\{M_t,\,t\in(-1,1)\}$. Hence we have obtained (1b).

{\it Second case: $H=0$}. In this case, equation \eqref{derivatives} becomes in
\begin{equation}\label{H}
\begin{split}
X(\sigma_{22})=\frac{1}{2}+\sigma_{22}^2-\sigma_{23}^2,\quad X(\sigma_{23})=2\sigma_{22}\sigma_{23},\\
\langle \nabla\sigma_{22},E_2\rangle=-\langle \nabla\sigma_{23},E_3\rangle,\quad \langle \nabla\sigma_{22},E_3\rangle=\langle \nabla\sigma_{23},E_2\rangle.
\end{split}
\end{equation}
Now,  if $\Delta=\sum_{i=1}^3(E_iE_i-\nabla_{E_i}E_i)$ is the Laplacian of the induced metric on $M$, from \eqref{H} we have that
\[
\Delta\sigma_{22}=\Delta\sigma_{23}=0,
\]
that is $\sigma_{22}$ and $\sigma_{23}$ are harmonic functions on $M$.

If $M$ is compact, then $\sigma_{22}$ and $\sigma_{23}$ are constant functions. Using \eqref{H} again we have two posibilities: $\sigma_{22}=0$ or $\sigma_{23}=0$. In the first case, $\sigma_{23}^2=1/2$ and hence we are again in the situation of the first case. So our hypersurface is  congruent to $M_0$ and we obtain (1b). In the second case ($\sigma_{23}=0$), if  $\gamma:\r\rightarrow M$ is an integral curve of $X$, which is defined on all $\r$, we can integrate $X(\sigma_{22})=1/2+\sigma_{22}^2$ along $\gamma$ and we get $\sigma_{22}(t)=(1/\sqrt{2})\tan (t/\sqrt{2}+a), \, a\in \r$, who is not defined in all $\r$. This is a contradiction and the hypersurface cannot be compact.

Also, if the scalar curvature is constant, then $\sigma_{22}^2+\sigma_{23}^2$ will be constant, and so, derivating with respect to $X$ and using
\eqref{H} we will obtain that
\[
0=\sigma_{22}(1+2\sigma_{22}^2+2\sigma_{23}^2),
\]
which implies that $\sigma_{22}=0$ and hence $\sigma_{23}^2=1/2$. This says that $M$ is locally congruent to $M_0$. Hence,  in this case we get that either $M$ is locally congruent to $M_0$ or $H=0$, $C=0$ and $M$ is not compact with non-constant scalar curvature. This implies  $(1c)$.

\vspace{0.5cm}

{\it Case (2): the scalar curvature $\rho$ is constant}.

In this case, as the scalar curvature $\rho=2+2(\sigma_{22}\sigma_{33}-\sigma_{23}^2)$ is constant, from \eqref{derivatives} it follows that
\begin{equation}\label{scalar}
(c_0-1)(c_0-\rho+3)\sigma_{22}=(c_0+1)(c_0+\rho-3)\sigma_{33}.
\end{equation}
Using again \eqref{derivatives} and derivating  equation \eqref{scalar} with respect to $X$ we obtain that
\[
(\rho-1)(\rho-3)+c_0^2=0.
\]
Hence there are two possible values of the scalar curvature: 
$\rho=2\pm\sqrt{1-c_0^2}$.

{\it First case: $\rho=2-\sqrt{1-c_0^2}$}.  
Puting this information in \eqref{scalar} it follows that
\begin{equation}\label{sigma_22,sigma_33}
\sqrt{1-c_0}\,\sigma_{22}=\sqrt{1+c_0}\,\sigma_{33}.
\end{equation}
As in this case 
$\sigma_{22}\sigma_{33}-\sigma_{23}^2=-\frac{\sqrt{1-c_0^2}}{2}$, \eqref{sigma_22,sigma_33} becomes in
\[
\frac{\sqrt{1-c_0}}{\sqrt{1+c_0}}\sigma_{22}^2-\sigma_{23}^2=-\frac{\sqrt{1-c_0^2}}{2}.
\]
Using the last equation of  \eqref{derivatives} in the above equation and taking into account \eqref{sigma_22,sigma_33}  we get 
\begin{eqnarray*}
\sigma_{23}E_2(\sigma_{23})-\sigma_{22}E_3(\sigma_{23})=0,\\
\frac{\sqrt{1-c_0}}{\sqrt{1+c_0}}\sigma_{22}E_2(\sigma_{23})-\sigma_{23}E_3(\sigma_{23})=0.
\end{eqnarray*}
Now, the only solution to this compatible homogeneous system is $E_2(\sigma_{23})=E_3(\sigma_{23})=0$.

But
\[
[E_2,E_3]=-2\sqrt{1-c_0^2}\sigma_{23}E_1,
\]
and so $0=[E_2,E_3](\sigma_{23})=-2\sigma_{23}X(\sigma_{23})=-2\sigma_{22}\sigma_{23}^2$.
As $\sigma_{23}$ can not have zeroes, we get that $\sigma_{22}=0$ and so $\sigma_{33}=0$. Going again to \eqref{derivatives} we get that $c_0=0$, and so $\rho=1$. This situation has been studied en Case 1, and we obtain (2b).

{\it Second case: $\rho=2+\sqrt{1-c_0^2}$}.  
As in the above case, putting the value of the scalar curvature in \eqref{scalar} it follows that
\[ 
\sqrt{1-c_0}(\sqrt{1+c_0}-\sqrt{1-c_0})\sigma_{22}=-\sqrt{1+c_0}(\sqrt{1+c_0}-\sqrt{1-c_0})\sigma_{33}.
\]
If $c_0\not=0$, from the above equation we get that $\sqrt{1-c_0}\,\sigma_{22}=-\sqrt{1+c_0}\,\sigma_{33}$ and using that $\sigma_{22}\sigma_{33}-\sigma_{23}^2=\frac{\sqrt{1-c_0^2}}{2}$ we have that
\[
-\big (\frac{\sqrt{1-c_0}}{\sqrt{1+c_0}}\sigma_{22}^2+\sigma_{23}^2\big )=\frac{\sqrt{1-c_0^2}}{2}
\]
which is impossible. Hence, in this second case, we get that $c_0=0$, $\rho=3$ and $\sigma_{22}\sigma_{33}-\sigma_{23}^2=\frac{1}{2}$. 

If $\{v,w\}$ is an orthonormal basis of a plane $\Pi\subset T_pM$, then the Gauss equation says that the curvature $K$ of $\Pi$ is given by
\[
K=\bar{R}(v,w,w,v)+ \sigma(v,v)\sigma(w,w)-\sigma(v,w)^2.
\]
Using the above information it is easy to check that 
\begin{gather*}
\bar{R}(v,w,w,v)=1/2-1/2(\langle v,E_2\rangle\langle w,E_3\rangle-\langle v,E_3\rangle\langle w,E_2\rangle)^2,\\
\sigma(v,v)\sigma(w,w)-\sigma(v,w)^2=\\=(\sigma_{22}\sigma_{33}-\sigma_{23}^2)(\langle v,E_2\rangle\langle w,E_3\rangle-\langle v,E_3\rangle\langle w,E_2\rangle)^2,
\end{gather*}
and so $K=1/2$. This means that $M$ has constant curvature $1/2$.

Now, we are going to see that $M$ is not complete. In fact, if $M$ is complete, as $M$ has constant positive curvature,  Myers' theorem  says that $M$ is  compact. On the other hand,  from \eqref{derivatives} we have that 
\[
X(\sigma_{22}-\sigma_{33})=2+(\sigma_{22}-\sigma_{33})^2.
\]
But Lemma~\ref{le:C,X} says that $\hbox{div}\,(X)=-(\sigma_{22}-\sigma_{33})$. So
\begin{gather*}
\hbox{div}((\hbox{div} X)X)=X(\hbox{div} X)+(\hbox{div} X)^2\\
=-2-(\sigma_{22}-\sigma_{33})^2+(\sigma_{22}-\sigma_{33})^2=-2,
\end{gather*}
and the divergence Theorem gives a contradiction. Hence $M$ is not complete.

Also, we are going to see that the mean curvature $H$ is not constant. In fact, if $H$ is constant, from Lemma~\ref{le:C,X} we get that 
\[
0=tr\, (P^{T}A^2)=3H(\sigma_{22}-\sigma_{33}).
\] 
As $H$ cannot be zero, because $\sigma_{22}\sigma_{33}-\sigma_{23}^2=1/2$, we obtain that $\sigma_{22}-\sigma_{33}=0$, which contradicts, using \eqref{derivatives},  the equation of its derivative with respect to $X$. Hence $M$ has not constant mean curvature, and we get (2c).
\end{proof}
As a consequence of  Corollay~\ref{co:Cconstant} we classify locally  the homogeneous hypersurfaces of $\s^2\times\s^2$.

\begin{corollary}\label{co:homogeneous}
Let $\Phi:M\rightarrow \s^2\times\s^2$ be an orientable  hypersurface . If $M$ is locally homogeneous, then $\Phi(M)$ is congruent to either an open subset of $\s^1(r)\times\s^2,\; r\in(0,1]$, or an open subset of $M_t,\; t\in(-1,1)$.
\end{corollary}
\begin{proof}
\end{proof}Let $N$ be a unit normal vector field. We fix a point $p_0\in M$. Then, as $M$ is locally homogeneous, for any $p\in M$ there exist open sets $p_0\in U_0$, $p\in U$ and an isometry $F$ of $\s^2\times\s^2$ such that $(F\circ\Phi)(U_0)=\Phi(U)$ and $F(\Phi(p_0))=\Phi(p)$. Then $N_p=\pm dF_{p_0}(N_{p_0})$ and $dF_{p_0}\circ P=\pm P\circ dF_{p_0}$. Hence $C(p)=\pm C(p_0)$ and as $M$ is connected, $C$ is constant. Also, it is clear that $F$ keeps the second fundamental form, and so the mean curvature and the scalar curvature of $M$ are also constant.  Hence Corollay~\ref{co:homogeneous}  follows from Corollay~\ref{co:Cconstant}.

\section{Isoparametric hypersurfaces}
In this section we classify the isoparametric hypersurfaces of $\s^2\times\s^2$. If $M$ is an isoparametric hypersurface of $\s^2\times\s^2$, then there exists an isoparametric function $F:\s^2\times\s^2\rightarrow\r$ such that $M=F^{-1}(t_0)$ for some regular value $t_0$. Then it is well-known that $M_t:=F^{-1}(t)$ are also hypersurfaces for $t\in(t_0-\delta,t_0+\delta)$, parallel to $M$ and with constant mean curvature. We start classifying  hypersurfaces of $\s^2\times\s^2$ satisfying this property.
\begin{theorem}\label{te:isoparametric}
Let $\Phi:M\rightarrow \s^2\times\s^2$ be an orientable hypersurface. If the parallel hypersurfaces $\Phi_t:M\rightarrow \s^2\times\s^2$, $t\in(-\epsilon,\epsilon),\,\Phi_0=\Phi$, have constant mean curvature, then $\Phi(M)$ is congruent either  to an open subset of $\s^1(r)\times\s^2,\; r\in(0,1]$, or to an open subset of $M_t,\; t\in(-1,1)$.
\end{theorem}
\begin{proof}
 We are going to consider the open subset of $M$ defined by $O=\{p\in M\,|\,C^2(p)<1\}$. If $O$ is empty, then $M$ has constant mean curvature and $C^2=1$, and hence Theorem~\ref{te:Cconstant} says that $\Phi(M)$ is congruent to an open subset of $\s^1(r)\times\s^2,\,r\in (0,1]$.

Now we suppose that $O$ is not empty. We write $\Phi=(\phi,\psi):O\rightarrow\s^2\times\s^2$, with $\phi,\psi:O\rightarrow \s^2$.
If $N=(N_1,N_2)$  is a unit normal vector field to $\Phi$, then the parallel hypersurfaces $\Phi_t:O\rightarrow \s^2\times\s^2$ are given by
\[
\Phi_t=(\exp_{\phi}tN_1,\exp_{\psi} tN_2),\quad t\in(-\epsilon,\epsilon),\quad \Phi_0=\Phi,
\]
where $\exp$ denotes the exponential map in $\s^2$. As $|N_1|^2=(1+C)/2$ and $|N_2|^2=(1-C)/2$, then $\Phi_t=(\phi_t,\psi_t)$ is defined by
\begin{gather*}
\phi_t=\cos(C^+t)\phi+(1/C^+)\sin(C^+t) N_1,\\
\psi_t=\cos(C^-t)\psi+(1/C^-)\sin(C^-t) N_2,
\end{gather*}
where $C^+=\sqrt{1+C}/\sqrt 2$ and $C^-=\sqrt{1-C}/\sqrt 2$.

Now, it is straightforward to check that $N^t=(N^t_1,N^t_2)$ defined by
\begin{gather*}
N^t_1=\cos(C^+t)N_1-C^+\sin(C^+t) \phi,\\
N^t_2=\cos(C^-t)N_2-C^-\sin(C^-t) \psi,
\end{gather*}
is a unit normal vector field to the hypersurface $\Phi_t$. Under these conditions it is easy to see that 
\[
C_t=C,\quad J_1N^t=J_1N,\quad J_2N^t=J_2N, \quad \forall t\in(-\epsilon,\epsilon).
\]
Hence, in $(\Phi_t)_*(TM_{|O})$ we can consider the orthonormal reference $\{E_i^t,\,i=1,2,3\}$ where
\[
E_1^t=\frac{X_t}{\sqrt{1-C^2}},\quad E_2^t=\frac{J_1N^t+J_2N^t}{\sqrt{2(1+C)}},\quad E_3^t=\frac{J_1N^t-J_2N^t}{\sqrt{2(1-C)}}.
\]  
Taking into account the above relations, if we denote $E_i^0=E_i$, it is clear that $E_2^t=E_2,, E_3^t=E_3,\,\forall t\in (-\epsilon,\epsilon)$.

Now, let $\{e_1,e_2,e_3\}$ be a local orthonormal reference on $O$ such that $\Phi_*(e_i)=E_i,\,i=1,2,3$. We denote $\sigma_{ij}=\langle \Phi_*(Ae_i),\Phi_*(e_j)\rangle$ the second fundamental form of $\Phi$ associated to the normal field $N$. Then, from a simply but long computation, we obtain that
\begin{gather*}
(\Phi_t)_*(e_i)=(\delta_{1i}-t\sigma_{1i}) E_1^t\\
+\big(\delta_{2i}\cos (C^+t)-\sigma_{2i}\frac{\sin(C^+t)}{C^+}\big)E^t_2\\
+\big(\delta_{3i}\cos (C^-t)-\sigma_{3i}\frac{\sin(C^-t)}{C^-}\big) E_3^t
\end{gather*}
and 
\begin{gather*}
(\Phi_t)_*(A^te_i)=\sigma_{1i} E_1^t\\
+\big(\sigma_{2i}\cos (C^+t)+\delta_{2i}C^+\sin(C^+t)\big) E_2^t\\
+\big(\sigma_{3i}\cos (C^-t)+\delta_{3i}C^-\sin(C^-t)\big) E_3^t,
\end{gather*}
where $A^t$ is the shape operator of $\Phi_t$ associated to the normal field $N^t$. 

If we denote $(\Phi_t)_*(e_i)=\sum_{j=1}^3Q_{ij}E_j^t$, it is clear from the above expressions that $(\Phi_t)_*(A^te_i)=-\sum_{j=1}^3Q'_{ij}E^t_j$, where $'$ stands for derivative with respect to $t$. Hence the induced metric $g_t$ on $O$ by the immersion $\Phi_t$ and the second fundamental form $\sigma^t$ of $\Phi_t$ are given by
\[
g_t=QQ^T,\quad \sigma^t=-Q(Q^T)',
\]
where $Q$ is the matrix $Q=(Q_{ij})$ and  $(\cdot)^T$ stands for the transpose. The mean curvatures of the immersions $\Phi_t$ are given by
\[
3H(t)=\traza\, (g_t^{-1}\sigma^t)=-\traza\,\big ( (Q^T)^{-1}Q^{-1}Q(Q^{T})'\big )=-\frac{(\det Q)'}{\det Q},
\] 
where $\traza$ stands for the trace and $\det$ stands for the determinant.
Hence $(\det Q)'=-3H(t)\det Q$ and as $Q(0)=Id$, an inductive argument says that
\[
\big (\frac{d^k\det Q}{dt^k}\big )(0),\quad k\geq 0,
\]
are constants functions on $O$.

From the definition, the determinant of $Q$ is given by
\begin{gather*}
\det Q=(1-t\sigma_{11})\cos(C^+t)\cos(C^-t)+(H_{23}-tK)\frac{\sin(C^+t)\sin(C^-t)}{C^+C^-}\\
+(-\sigma_{22}+tH_{12})\frac{\sin(C^+t)\cos(C^-t)}{C^+}+(-\sigma_{33}+tH_{13})\frac{\cos(C^+t)\sin(C^-t)}{C^-},
\end{gather*}
where $H_{ij}=\sigma_{ii}\sigma_{jj}-\sigma_{ij}^2$ and $K=\det A$ is the Gauss-Kronecker curvature of $M$.

Now, computing the Taylor serie of the function $\det Q$ around $t=0$ and from a very long computation we get that
\begin{gather*}
\det Q=1-3H\,t+\frac{3-\rho}{2}\,t^2+ \frac{9H-6K-(1+C)\sigma_{22}-(1-C)\sigma_{33}}{3!}\,t^3\\
+\frac{2(3-\rho)-C^2-2((1-C)H_{12}+(1+C)H_{13})}{4!}\,t^4\\
+ \frac{-5(2-C^2)3H+20K+4(1+C-C^2)\sigma_{22}+4(1-C-C^2)\sigma_{33}}{5!}\,t^5\\
+\frac{4(\rho-3)+(5-\rho)C^2+4((1-C)(4+C)H_{12}+(1+C)(4-C)H_{13})}{6!}\,t^6+\dots
\end{gather*}
As all the coefficients of the above serie are constant functions, we get that not only the mean curvature but also the scalar curvature $\rho$ of $M$ is constant.

Now we work on the open subset $V=\{p\in O\,|\,C(p)\not=0\}$. Taking into account that the coefficients corresponding to $t^4$ and $t^6$ are constant and that $\rho$ is also constant, we obtain on $V$ that
\[
H_{12}=\frac{(1+C)R_2(C)}{8C(1-C^2)},\quad H_{13}=\frac{(1-C)R_3(C)}{8C(1-C^2)},
\]
where $R_2$ and $R_3$ are non-trivial polynomials of degree $3$ in $C$ with constant coefficients. Computing the term of the Taylor serie corresponding to $t^8$, we have that 
\begin{gather*}
-8(\rho-3)+4(\rho-4)C^2+(-37+11C+17C^2-C^3)H_{12}\\+(-37-11C+17C^2+C^3)H_{13}=\lambda, 
\end{gather*}
for certain constant $\lambda$. Using the above expressions of $H_{12}$ and $H_{13}$ in this equation, we finally prove that $C$ satisfies a non trivial polynomial of degree 7 with constant coefficients. This means that $C$ is constant on each connected component of the open set $V$.

Hence the function $C$ on the connected hypersurface $M$ takes only a discrete number of values. This means that the function $C$ is constant and the result follows from  Corollay~\ref{co:Cconstant}.
\end{proof}
\begin{corollary}\label{co:isoparametric}
Let $\Phi:M\rightarrow\s^2\times\s^2$ be an isoparametric hypersurface. Then $M$ is congruent either to $\s^1(r)\times\s^2,\; r\in(0,1]$, or to $M_t,\; t\in(-1,1)$.
\end{corollary}

\section{Hypersurfaces with constant principal curvatures}

In this section we are going to study orientable hypersurfaces of $\s^2\times\s^2$ with constant principal curvatures. 

\subsection{Hypersurfaces with one or two constant principal curvatures}
In the following result we classify locally the orientable hypersurfaces of $\s^2\times\s^2$ with at most two constant principal curvatures.
\begin{theorem}\label{te:2curvaturas}
Let $\Phi:M\rightarrow\s^2\times\s^2$ be an orientable hypersurface of $\s^2\times\s^2$. If $M$ has at most two constant principal curvatures, then $\Phi(M)$ is either an open subset of $\s^1\times\s^2$ ($M$ is totally geodesic), or $\Phi(M)$ is an open subset of $\s^1(r)\times\s^2,\;r\in(0,1) $ (if $M$ has two constant principal curvatures).
\end{theorem}
\begin{remark}
When the hypersurface has only one principal curvature, not necessarely constant, i.e.,when  the hypersurface is umbilical, it is easy to conclude  that the mean curvature is constant and hence the hypersurface is totally geodesic.
\end{remark}
\begin{proof}
If $M$ has only one constant principal curvature, then $M$ is an umbilical hypersurface with constant mean curvature. Then in [TU], Proposition 1, it was proved that $M$ is congruent to an open subset of $\s^1\times\s^2$.

Now, we suppose that  $M$ has two constant principal curvatures. Let $\lambda_1,\lambda_2$ be the corresponding different principal curvatures having  $\lambda_1$ multiplicity one and  $\lambda_2$ multiplicity two. 

Under these conditions, let $E_1$ be a unit vector field on $M$ such that $AE_1=\lambda_1 E_1$. Then,  we have that the second fundamental form and its covariant derivative are given by
\begin{equation}\label{2curvaturas}
\begin{split}
\sigma(V,W)=\lambda_2 \langle V,W\rangle+(\lambda_1-\lambda_2)\langle V,E_1\rangle\langle W, E_1\rangle,\\
(\nabla\sigma)(Z,V,W)=(\lambda_1-\lambda_2)\big(\langle W,E_1\rangle\langle V,\nabla_ZE_1\rangle+\langle V,E_1\rangle\langle W,\nabla_ZE_1\rangle\big),
\end{split}
\end{equation}
for any vector fields $V,W,Z$ on $M$.

Using \eqref{2curvaturas}, Codazzi equation and the fact that $M$ has constant mean curvature, we obtain 
\begin{eqnarray*}
0=\sum_{i=1}^3(\nabla\sigma)(Z,e_i,e_i)=\sum_{i=1}^3(\nabla\sigma)(e_i,Z,e_i)=\\
(\lambda_1-\lambda_2)\big( \langle E_1,Z\rangle\hbox{div}\,E_1+\langle\nabla_{E_1}E_1,Z\rangle\big),
\end{eqnarray*}
for any vector field $Z$ on $M$, being $\{e_1,e_2,e_3\}$  an orthonormal reference on $M$ and $\hbox{div}$ the divergence operator.

As $|E_1|^2=1$, from the last equation we get that the vector field $E_1$ is a geodesic vector field with zero divergence, i.e., 
\begin{equation}\label{div}
\hbox{div}\,E_1=0 \quad\hbox{and}\quad\nabla_{E_1}E_1=0 .
\end{equation}
Using \eqref{div} in the second equation of \eqref{2curvaturas} we obtain
\[
(\nabla\sigma)(E_1,V,W)=0,\quad (\nabla\sigma)(V,E_1,W)=(\lambda_1-\lambda_2)\langle\nabla_VE_1,W\rangle,
\]
and hence, the Codazzi equation says that
\begin{equation}\label{derivadaE_1}
\nabla_VE_1=\frac{\langle X,V\rangle P^{T}E_1-\langle X,E_1\rangle P^{T}V}{2(\lambda_1-\lambda_2)},
\end{equation}
for any vector field $V$ tangent to $M$, where $P^{T}$ denotes the tangential component of $P$, i.e., $P^{T}Z=PZ-\langle X,Z\rangle N$.

Now, from \eqref{div} and \eqref{derivadaE_1} and making a direct computation, we obtain that the Ricci curvature of $E_1$ is given by
\[
 \Ric(E_1)=-\frac{\langle X,E_1\rangle^2}{2(\lambda_1-\lambda_2)^2}.
 \]

 But, on the other hand, the Gauss equation says that 
 \[
 \Ric(E_1)=\frac{1}{2}+2\lambda_1\lambda_2+\frac{\langle X, E_1\rangle^2}{2}-\frac{C\langle PE_1,E_1\rangle}{2},
 \] 
 and hence finally we get that 
 \begin{equation}\label{ricci}
 1+4\lambda_1\lambda_2=C\langle PE_1,E_1\rangle-(1+\frac{1}{(\lambda_1-\lambda_2)^2})\langle X,E_1\rangle^2.
 \end{equation}

Now we proceed as follows. From Lemma~\ref{le:C,X},(1) and \eqref{derivadaE_1} it is clear that
\begin{gather*}
E_1(C)=-2\lambda_1\langle X,E_1\rangle,\quad E_1(\langle PE_1,E_1\rangle)=2\lambda_1\langle X,E_1\rangle,\\ E_1(\langle X,E_1\rangle)=\lambda_1(C-\langle PE_1,E_1\rangle).
\end{gather*}
Hence, taking derivatives in \eqref{ricci} with respect to $E_1$ and taking into account the above expressions, we get that
\[
0=\langle X,E_1\rangle\big(\langle PE_1,E_1\rangle-C\big).
\]
Taking derivatives  again this equation with respect to $E_1$ and using again the above expressions we obtain $4\langle X,E_1\rangle^2=(\langle PE_1,E_1\rangle -C)^2$, and hence the above equation implies that
\[
\langle X,E_1\rangle=0,\quad \langle PE_1,E_1\rangle=C.
\]
Now, \eqref{ricci} becomes in $C^2=1+4\lambda_1\lambda_2.$ This means that $C$ is a constant function and the result follows from Corollary~\ref{co:Cconstant}.
\end{proof}

\subsection{Hypersurfaces with three constant principal curvatures}

From now on,  let $\Phi:M\rightarrow\s^2\times\s^2$ be an orientable hypersurface with three different constant principal curvatures $\{\lambda_1,\lambda_2,\lambda_3\}$. Then there exists a trivialization of $M$ by unit vector fields $\{ E_1,E_2,E_3\}$ with $AE_i=\lambda_i E_i, \,i=1,2,3$. These principal curvatures are the roots of the polynomial 
\begin{equation}\label{polynomial}
\lambda^3-3H\lambda^2+\frac{\rho-2}{2}\lambda-K=0.
\end{equation}
 For simplicity, we will denote 
\begin{gather*}
\sigma_{ij}=\sigma(E_i,E_j),\quad \nabla\sigma_{ijk}=(\nabla\sigma)(E_i,E_j,E_k),\quad P_{ij}=\langle PE_i,E_j\rangle, \\b_i=\langle X, E_i\rangle, \quad\Gamma_{ij}^k=\langle\nabla_{E_i}E_j,E_k\rangle. 
\end{gather*}
\begin{lemma}\label{le:simbolos}
The above functions satisfy the following properties
\begin{enumerate}
\item $\Gamma_{ij}^k+\Gamma_{ik}^j=0, \quad\quad 1\leq i,j,k\leq 3$.
\item $\nabla\sigma_{ijk}+(\lambda_k-\lambda_j)\Gamma_{ij}^k=0,\quad \quad 1\leq i,j,k\leq 3$.
\item $(\lambda_k-\lambda_j)\Gamma_{ij}^k-(\lambda_k-\lambda_i)\Gamma_{ji}^k=\frac{1}{2}\big(b_j P_{ik}-b_i P_{jk}\big),\quad\, 1\leq i,j,k\leq 3$
\item $\Gamma_{ii}^j=\frac{b_i P_{ij}-b_j P_{ii}}{2(\lambda_i-\lambda_j)},\quad\quad 1\leq i,j\leq 3,\quad i\not=j.$
\item $ (\lambda_i-\lambda_j)\Gamma_{ii}^j=-(\lambda_k-\lambda_j)\Gamma_{kk}^j,\quad\quad 1\leq i,j\leq 3,\quad  i\not=j\not=k.$
\end{enumerate}
\end{lemma}
\begin{proof}
(1) is trivial.  Taking derivatives in $\sigma_{jk}=\lambda_j\delta_{jk}$ with respect to $E_i$ we get (2). From Codazzi equation and (2) we prove (3). Taking $k=i$ in (3) and using (1) we get (4). Finally from \eqref{P} and (4) it follows (5).
\end{proof}
\begin{lemma}\label{le:Lambda}
The functions $\Lambda_i=b_i^2-CP_{ii},\,i=1,2,3,$ satisfy the following equation
\[
B\Lambda=B_0+D,
\]
where
\[
B=\left( \begin{array}{ccc} \lambda_2 & -\lambda_1 & 2(\lambda_1-\lambda_2) \\ \lambda_3 & 2(\lambda_1-\lambda_3) & -\lambda_1 \\ 2(\lambda_2-\lambda_3) &   \lambda_3 & -\lambda_2  \end{array} \right),\quad
\Lambda=\left( \begin{array}{c} \Lambda_1 \\ \Lambda_2\\ \Lambda_3 \end{array} \right) 
\]
\[
B_0=\left( \begin{array}{c} (\lambda_1-\lambda_2)(1+2\lambda_1\lambda_2) \\ (\lambda_1-\lambda_3)(1+2\lambda_1\lambda_3)\\ (\lambda_2-\lambda_3)(1+2\lambda_2\lambda_3) \end{array} \right),\quad
D=\left( \begin{array}{c} D_{12} \\ D_{13}\\ D_{23} \end{array} \right) 
\]
and $D_{ij},\,i\not= j$ are defined by
\[
D_{ij}=4(\lambda_i-\lambda_j)((\Gamma_{ii}^j)^2+(\Gamma_{jj}^i)^2)+4(\lambda_k-\lambda_j)(\Gamma_{ij}^k)^2-4(\lambda_k-\lambda_i)(\Gamma_{ji}^k)^2,
\]
with $k\not= i\not= j$.

\end{lemma}
\begin{remark}
We observe that Lemma~\ref{le:Lambda} provides that the functions $\Lambda_i$ satisfy a compatible linear system because  
$$\det B=-6(\lambda_1-\lambda_2)(\lambda_1-\lambda_3)(\lambda_2-\lambda_3)\not=0.$$ 

\end{remark}
\begin{proof}
For $i,j\in\{1,2,3\},\, i\not= j$, let $K_{ij}$ be the sectional curvature of the plane spanned by $\{E_i,E_j\}$. Then from the definition of the sectional curvature we have that
\begin{eqnarray*}
K_{ij}&=&E_i(\Gamma_{jj}^i)+E_j(\Gamma_{ii}^j)-(\Gamma_{ii}^j)^2-(\Gamma_{jj}^i)^2-\Gamma_{ii}^k\Gamma_{jj}^k\\
&-&\Gamma_{ij}^k(\Gamma_{jk}^i+\Gamma_{kj}^i)+\Gamma_{ji}^k\Gamma_{kj}^i,
\end{eqnarray*}
where $k\not= i$ and $k\not= j$.

Now, using Lemma~\ref{le:simbolos} and Lemma~\ref{le:C,X}, (1) and from a straightforward computation it is easy to conclude that

\begin{gather*}
K_{ij}=-2((\Gamma_{ii}^j)^2+(\Gamma_{jj}^i)^2)-\frac{1}{2}(P_{ii}P_{jj}-P_{ij}^2)\\
+\frac{1}{2(\lambda_j-\lambda_i)}(\lambda_i\Lambda_j-\lambda_j\Lambda_i)\\
+\frac{2}{(\lambda_j-\lambda_i)}((\lambda_k-\lambda_j)(\Gamma_{ij}^k)^2-(\lambda_k-\lambda_i)(\Gamma_{ji}^k)^2),\quad k\not=i\not=j.
\end{gather*}

Finally using the Gauss equation, $2K_{ij}=1+P_{ii}P_{jj}-P_{ij}^2+2\lambda_i\lambda_j$, and \eqref{P} in the above equation we obtain that
\[
\lambda_j \Lambda_i-\lambda_i \Lambda_j+2(\lambda_i-\lambda_j) \Lambda_k=(\lambda_i-\lambda_j)(1+2\lambda_i\lambda_j)+D_{ij},\quad i\not=j\not=k,
\] 
and  the Lemma is proved.
\end{proof}

\begin{lemma}\label{le:C=1}
If there exists a point  $p$ on $M$  with $C^2(p)=1$, then
\begin{gather*}
\Lambda_i(p)=\frac{2}{3}(\lambda_i^2-3H\lambda_i+\rho-\frac{1}{2})\\
+\frac{12|\nabla\sigma|^2(p)}{(\det B)^2}\{2(\rho-2-6H^2)\lambda_i^2+3(3K+12H^3-\frac{5}{2}H(\rho-2))\lambda_i\\+(\rho-2)(\rho-2-3H^2)-18HK\}, \quad i\in\{1,2,3\},
\end{gather*}
and 
\[
1-2\rho=18\{(\rho-2)(\rho-2-3H^2/2)-27HK\}\frac{|\nabla\sigma|^2(p)}{(\det B)^2}.
\]
\end{lemma}
\begin{proof}
Suppose now that $p$ is a point with $C^2(p)=1$. Then $X_p=0$ and so $b_i(p)=0,\,1\leq i\leq 3$. Using Lemma~\ref{le:simbolos} we have that $$D_{ij}(p)=\frac{4(\lambda_j-\lambda_i)(\lambda_k-\lambda_j)}{\lambda_k-\lambda_i}(\Gamma_{ij}^k)^2(p),\quad 1\leq i,j,k\leq 3,\,i<j, \,k\not=i,\,k\not= j.$$
On the other hand, using again Lemma~\ref{le:simbolos} it is easy to check that
\[
|\nabla\sigma|^2(p)=6(\lambda_k-\lambda_j)^2(\Gamma_{ij}^k)^2(p),\quad i\not= j\not= k.
\]
From the last two equatios it follows that 
\[
D(p)=\frac{4|\nabla\sigma|^2(p)}{\det B}\left( \begin{array}{c} (\lambda_1-\lambda_2)^2 \\ -(\lambda_1-\lambda_3)^2\\ (\lambda_2-\lambda_3)^2 \end{array} \right) 
\]
Now, solving the system of Lemma~\ref{le:Lambda} and using \eqref{polynomial}  we get the first part. 
The last assertion in the Lemma follows from the fact that $\sum_{i=1}^3\Lambda_i=1$, which easily follows from  Lemma~\ref{le:simbolos}. 
\end{proof}

\begin{theorem}\label{te:3}
Let $\Phi:M\rightarrow \s^2\times\s^2$ be an orientable hypersurface with three constant principal curvatures.
\begin{enumerate}
\item If the scalar curvature $\rho$ satisfies $2\rho\not=1$ and the Gauss-Kronecker curvature $K=0$, then $C^2<1$.
\item If $p_0$ is a critical point  of the function $C$, then either $C^2(p_0)=1$  or $C(p_0)=0$. Moreover, in the second case, the Gauss-Kronecker curvature $K=0$ and the scalar curvature  $\rho=1$.
\end{enumerate}

\end{theorem}
\begin{proof}
First we are going to prove part (1). 

 Suposse that there exists a point $p_0\in M$ such that $C^2(p_0)=1$ and without loss of generality we can take  $C(p_0)=1$.  Then we are going to get a contradiction.

As $K=0$, one of the principal curvatures is zero, for instance $\lambda_1=0$. Let $\gamma:(-\epsilon,\epsilon)\rightarrow M$ be the integral curve of $E_1$ with $\gamma(0)=p_0$. Then

\[
C'(t)=\langle (\nabla C)(t),E_1(t)\rangle=-2\langle AX(t),E_1(t)\rangle=0,
\]
and hence  $C(t)=1$ for all $t\in(-\epsilon,\epsilon)$. Now, as the Gauss-Kronecker curvature $K=0$, from  Lemma~\ref{le:C=1} it follows that
\[
(1-2\rho)(\det B)^2=18(\rho-2)(\rho-2-3H^2/2)|\nabla\sigma|^2(t).
\]
As $1-2\rho\not=0$, we have that $|\nabla\sigma|^2(t)=6(\lambda_3-\lambda_2)^2(\Gamma_{12}^3)^2(t)$ is a non-null constant function and that $\rho\not=2$. Hence, from Lemma~\ref{le:C=1},  the functions $\Lambda_i(t)=-P_{ii}(t)$ are also constant. So
\[
0=P_{22}'(t)=2\Gamma_{12}^3(t)P_{23}(t),
\]
which implies that $P_{23}(t)=0$. Derivating again
\[
0=P_{23}'(t)=\Gamma_{12}^3(t)(P_{33}(t)-P_{22}(t)),
\]and so $P_{22}(t)=P_{33}(t)$. Using all the information about the $P_{ij}$ in \eqref{P} we obtain that

\[
P_{11}(t)=-P_{22}(t)=-P_{33}(t)=1.
\]
Using the first part of  Lemma~\ref{le:C=1}, and as $\Lambda_2(t)=\Lambda_3(t)$ and $3H=\lambda_2+\lambda_3$ we get that $$18H(\rho-2)\frac{|\nabla\sigma|^2(t)}{(\det B)^2}=0,$$ which implies that $H=0$ because $\rho\not=2$. 

 Now, as $\Lambda_1(t)=-1$, $H=0$ and $\lambda_1=0$,  first part of Lemma~\ref{le:C=1} says that
 \[
 -(\rho+1)=18(\rho-2)^2\frac{|\nabla\sigma|^2(t)}{(\det B)^2},
 \]
 that joint with the information given in the second part of  Lemma~\ref{le:C=1}
 \[
 1-2\rho =18(\rho-2)^2\frac{|\nabla\sigma|^2(t)}{(\det B)^2}
 \]
 imply that $\rho=2$, which is a contradiction. This proves (1).

\vspace{0.3cm}

Now, we are going to prove part (2).

Let $p_0$ be a point of $M$ such that $(\nabla C)(p_0)=0$ and $C^2(p_0)<1$. We are going to prove that  $C(p_0)=0$. 

In this case, $AX(p_0)=0$ and as $|X|^2=1-C^2(p_0)>0$, one of the three principal curvatures is zero. For instance $\lambda_1=0$ and so $X(p_0)=\sqrt{1-C^2(p_0)}E_1(p_0)$. 
Now, let $\gamma:(-\epsilon,\epsilon)\rightarrow M$ be the integral curve of $E_1$ with $\gamma(0)=p_0$. As 
\[
C'(t)=\langle (\nabla C)(t), E_1(t)\rangle=-2\langle AX(t),E_1(t)\rangle =0,
\]
because $AE_1=0$, then $C$ is constant along $\gamma$ and $(\nabla C)(t)=\alpha(t)E_2(t)+\beta(t)E_3(t)$. Hence, derivating with respect to $t$ we have that

\begin{gather*}
(\nabla^2C)(E_1,E_2)(t)=\alpha'(t)-\Gamma_{12}^3(t)\beta(t),\\ (\nabla^2C)(E_1,E_3)(t)=\beta'(t)+\Gamma_{12}^3(t)\alpha(t).
\end{gather*}
On the other hand, from  Lemma~\ref{le:C,X} we obtain that
\begin{gather*}
(\nabla^2C)(E_1,E_2)=-2\lambda_2\Gamma_{22}^1\langle X,E_2\rangle+2\lambda_3\Gamma_{21}^3\langle X,E_3\rangle,\\
(\nabla^2C)(E_1,E_3)=2\lambda_2\Gamma_{31}^2\langle X,E_2\rangle-2\lambda_3\Gamma_{33}^1\langle X,E_3\rangle.
\end{gather*}
But, by definition $\alpha(t)=-2\lambda_2\langle X,E_2\rangle(t)$ and  $\beta(t)=-2\lambda_3\langle X,E_3\rangle(t)$, and so the above information about the Hessian of $C$ says that $\alpha$ and $\beta$ satisfy the following ODE system
\begin{gather*}
\alpha'(t)=\Gamma_{22}^1(t)\alpha(t)+(\Gamma_{12}^3(t)-\Gamma_{21}^3(t))\beta(t)\\
\beta'(t)=-(\Gamma_{12}^3(t)+\Gamma_{31}^2(t))\alpha(t)+\Gamma_{33}^1(t)\beta(t).
\end{gather*}
As $(\nabla C)(0)=0$, i.e. $\alpha(0)=\beta(0)=0$, it follows that  $\alpha=\beta=0$ is the only solution to the above system with that initial conditions. Hence, $(\nabla C)(t)=0,\, \forall t$, and so 
\[
X(t)=\sqrt{1-C^2(p_0)}E_1(t).
\]
This relation between $X$ and $E_1$ joint with \eqref{P} and Lemma~\ref{le:simbolos} imply that on the curve $\gamma$ we have the following relations
\begin{gather*}
b_1(t)=\sqrt{1-C^2(p_0)},\quad b_2(t)=b_3(t)=0, \\
\Lambda_1(t)=0,\quad \Lambda_2(t)+\Lambda_3(t)=0,\quad P_{12}(t)=P_{13}(t)=0, \\
\Gamma_{11}^2(t)=\Gamma_{11}^3(t)=\Gamma_{22}^3(t)=\Gamma_{33}^2(t)=0,\quad\lambda_2\Gamma_{32}^1(t)=\lambda_3\Gamma_{23}^1(t).
\end{gather*}

Now, using the above relations and Lemma~\ref{le:simbolos}, the linear system of Lemma~\ref{le:Lambda}  becomes in 
\begin{equation}\label{newsystem}
\begin{split}
2\lambda_2(\Gamma_{22}^1)^2(t)+2(\lambda_2-\lambda_3)(\Gamma_{12}^3)^2(t)+2\lambda_3(\Gamma_{21}^3)^2(t)=-\lambda_2(1+\Lambda_2(t)),\\
-\frac{2\lambda_2^2}{\lambda_3}(\Gamma_{22}^1)^2(t)+2(\lambda_2-\lambda_3)(\Gamma_{12}^3)^2(t)-\frac{2\lambda_3^2}{\lambda_2}(\Gamma_{21}^3)^2(t)=\lambda_3(1-\Lambda_2(t)),\\
\frac{-4\lambda_3}{\lambda_2}(\Gamma_{21}^3)^2(t)=1-2\lambda_2\lambda_3+\frac{\lambda_3+\lambda_2}{\lambda_2-\lambda_3}\Lambda_2(t).
\end{split}
\end{equation}
This means that the functions $(\Gamma_{22}^1)^2(t), (\Gamma_{12}^3)^2(t)$ and $(\Gamma_{21}^3)^2(t)$ satisfy  a linear system whose determinant is  $-16(\lambda_3-\lambda_2)(\lambda_3+\lambda_2)$.

In what follows we are going to consider two cases:

{\it First case: $H=\lambda_2+\lambda_3=0$}. In this case, \eqref{newsystem} becomes in 
\begin{equation}\label{H=0}
\begin{split}
\Lambda_2(t)=0,\\
4(\Gamma_{21}^3)^2(t)=1+2\lambda_2^2,\\
2(\Gamma_{22}^1)^2(t)+4(\Gamma_{12}^3)^2(t)=-\frac{1}{2}+\lambda_2^2.
\end{split}
\end{equation}
Firstly we prove that $C(p_0)=0$. If not, as $\Lambda_2(t)=B_2^2(t)-C(p_0)P_{22}(t)=-C(p_0)P_{22}(t)$, it follows from \eqref{H=0} that $P_{22}(t)=0$ and so $P_{23}(t)^2=1$. Hence
\[
0=P_{22}'(t)=2\Gamma_{12}^3(t)P_{23}(t),
\]
which means that $\Gamma_{12}^3(t)=0$. But, from Lemma~\ref{le:simbolos},  $\Gamma_{22}^1(t)=0$ and so \eqref{H=0} says that $\lambda_2^2=1/2$ and $(\Gamma_{21}^3)^2(t)=1/2$. Finally, from  Lemma~\ref{le:simbolos} it follows that $\lambda_2\Gamma_{21}^3(t)=-b_1P_{23}(t)/2$, which implies using the above information that $b_1^2=1$. This is a contradiction because $b_1^2=1-C^2(p_0)$ and we are assuming that $C(p_0)\not=0$.

Secondly we prove  that the scalar curvature $\rho=1$. From  Lemma~\ref{le:simbolos} we have that
\begin{equation}\label{minima}
P_{22}(t)=-2\lambda_2\Gamma_{22}^1(t),\quad P_{23}(t)=4\lambda_2\Gamma_{12}^3(t)-2\lambda_2\Gamma_{21}^3(t).
\end{equation}
Now using \eqref{H=0}, \eqref{minima} and the fact that $P_{22}^2(t)+P_{23}^2(t)=1$, we obtain that $\Gamma_{12}^3(t)$ satysfies the following non-trivial second order equation 
\[
8\lambda_2^2(\Gamma_{12}^3)^2(t)-16\lambda_2^2\Gamma_{21}^3\Gamma_{12}^3(t)+4\lambda_2^4-1=0.
\]
As  from \eqref{H=0} the function $\Gamma_{21}^3(t)$ is constant, the coefficients of the above polynomial are constant, and so $\Gamma_{12}^3(t)$ is also constant. Now \eqref{H=0} says that  $\Gamma_{22}^1(t)$ is also a constant function and from  \eqref{minima} we get that $P_{22}(t)$ and $P_{23}(t)$ are constant functions too. Hence
\[
0=P_{22}'(t)=2\Gamma_{12}^3P_{23},\quad 0=P_{23}'(t)=-\Gamma_{12}^3P_{33}+\Gamma_{13}^2P_{22}=-2\Gamma_{12}^3P_{22}.
\]
So $\Gamma_{12}^3(t)=0$ and from the above equation  it follows that $\lambda_2^2=1/2$, i.e., the scalar curvature $\rho=1$.

{\it Second case: $H\not=0$}. In this case the system \eqref{newsystem} is compatible and it is easy to check that
\begin{equation}\label{nominima}
\frac{2\lambda_2}{\lambda_3}(\Gamma_{22}^1)^2(t)=-\frac{1+2\lambda_2\lambda_3}{2}+\frac{\lambda_3^2+\lambda_2^2-6\lambda_3\lambda_2}{2(\lambda_3^2-\lambda_2^2)}\Lambda_2(t).
\end{equation}
Firstly we prove that $C(p_0)=0$. If not, from Lemma~\ref{le:simbolos}, it follows that
\[
\Gamma_{22}^1(t)=-\frac{b_1C(P_0)P_{22}(t)}{2\lambda_2C(P_0)}=\frac{b_1\Lambda_2(t)}{2\lambda_2C_0}.
\]
Putting this information in \eqref{nominima}, it follows that $\Lambda_2(t)$ satisfies a  non-trivial polynomial of degree two with constant coefficients. This means that $\Lambda_2(t)$ is a constant function and hence $P_{22}(t)$ and $P_{23}(t)$ are also constant. Using that $P_{22}'(t)=P_{23}'(t)=0$, like in the above minimal case, we get that $\Gamma_{12}^3(t)=0$. But then, the solution of  $\Gamma_{12}^3(t)$ which provides \eqref{newsystem} implies that 
\begin{equation}\label{Lambda2}
\Lambda_2(t)=\frac{(\lambda_3+\lambda_2)(1+2\lambda_3\lambda_2)}{\lambda_3-\lambda_2}.
\end{equation} 
From \eqref{nominima} and \eqref{Lambda2} it follows that 
\[
(\Gamma_{22}^1)^2(t)=-\frac{\lambda_3^2(1+2\lambda_3\lambda_2)}{(\lambda_3-\lambda_2)^2}=-\frac{\lambda_3^2}{\lambda_3^2-\lambda_2^2}\Lambda_2(t).
\]
But we know that 
\[
(\Gamma_{22}^1)^2(t)=\frac{(1-C^2(p_0))\Lambda_2^2}{4\lambda_2^2C^2(p_0)},
\]
which implies that
\[
\Lambda_2\big(\frac{(1-C^2(p_0))\Lambda_2}{4\lambda_2^2C^2(p_0)}+\frac{\lambda_3^2}{\lambda_3^2-\lambda_2^2}\big)=0.
\]
Hence we have two possibilities. If $\Lambda_2=0$, then $P_{22}=0$ and from \eqref{Lambda2} it follows that $\lambda_2\lambda_3=-1/2$ and $\Gamma_{22}^1=0$. Now \eqref{newsystem} says that $\Gamma_{12}^3=0$ and $\Gamma_{21}^3=\lambda_2^2$. Finally, from Lemma~\ref{le:simbolos}, (3) we get that $1-C^2(p_0)=1$ which is a contradiction.

On the other hand, if $\frac{(1-C^2(p_0))\Lambda_2}{4\lambda_2^2C^2(p_0)}+\frac{\lambda_3^2}{\lambda_3^2-\lambda_2^2}=0$, then \eqref{Lambda2} implies that
\[
C^2(p_0)=\frac{(\lambda_3+\lambda_2)^2(1+2\lambda_3\lambda_2)}{(\lambda_3+\lambda_2)^2+2\lambda_3\lambda_2(\lambda_3^2+\lambda_2^2)}>1,
\]
which is also  a contradiction. Hence we have proved that $C(p_0)=0$.

Secondly we prove that the scalar curvature $\rho=1$.  In this case, as $$\Lambda_2(t)=b_2^2(t)-C(p_0)P_{22}(t)=0,$$  \eqref{newsystem} becomes in the following equations
\begin{gather*}
(\Gamma_{21}^3)^2(t)=\frac{\lambda_2(2\lambda_2\lambda_3-1)}{4\lambda_3},\\
-2\lambda_2(\Gamma_{22}^1)^2(t)+2(\lambda_3-\lambda_2)(\Gamma_{12}^3)^2(t)=\lambda_2(\lambda_2\lambda_3+1/2),\\
\frac{2\lambda_2^2}{\lambda_3}(\Gamma_{22}^1)^2(t)+2(\lambda_3-\lambda_2)(\Gamma_{12}^3)^2(t)=-\lambda_3(\lambda_2\lambda_3+1/2),
\end{gather*}
which implies that
\[
(\Gamma_{22}^1)^2(t)=\frac{-\lambda_3(\lambda_2\lambda_3+1/2)}{2\lambda_2},\quad (\Gamma_{12}^3)^2(t)=-\frac{(\lambda_2\lambda_3+1/2)}{2}.
\]
As $(\Gamma_{22}^1)^2(t)$ and $(\Gamma_{12}^3)^2(t)$ are non-negative functions, the above equations say that $\lambda_2\lambda_3+1/2=0$ and so the scalar curvature of $M$ is $1$.
\end{proof}

\begin{corollary}\label{co:compact}
Let $\Phi:M\rightarrow\s^2\times\s^2$ be an orientable compact hypersurface with three different constant principal curvatures. If the scalar curvature $\rho$ satisfies $2\rho\not=1$ and the Gauss-Kronecker curvature $K=0$, then $\Phi(M)$ is congruent to $M_t$ for some $t\in(-1,1)$.
\end{corollary}
\begin{proof}
From Theorem~\ref{te:3}, (1),  the function $C$ satisfies $C^2<1$. As $M$ is compact, from Theorem~\ref{te:3}, (2),  the maximum and the minimum of $C$ is zero. So $C\equiv 0$ and the resul follows from Corollary~\ref{co:Cconstant}.
\end{proof}

\end{document}